\documentclass[a4paper,11pt]{article}
\usepackage{amssymb,enumerate}
\usepackage{amsmath,amsthm}
\usepackage{amsfonts}
\usepackage{graphicx}
\usepackage{hyperref}

\newcounter{case}

\renewcommand{\thecase}{\arabic{case}}

\usepackage{color}
\newtheorem{theorem}{Theorem}[section]

\newtheorem{lemma}[theorem]{Lemma}

\newtheorem{problem}[theorem]{Problem}
\newtheorem{conjecture}[theorem]{Conjecture}

\newtheorem{remark}[theorem]{Remark}

\newtheorem*{theorem3}{Theorem~\ref{thm:mainTHM} (restated)}

\def\ov{\overline}
\newcommand{\Aut}{\hbox{{\rm Aut}}}

\newcommand{\ZZ}{\mathbb{Z}}

\newcommand{\Cay}{\mathrm{Cay}}

\theoremstyle{definition}

\begin{document}

\begin{center}
{\bf \Large Canonical double covers of circulants} \\ [+4ex]
Blas Fernandez{\small$^{~a}$} and Ademir Hujdurovi\'c{\small$^{~a,b}$}\\ [+2ex]
{\it \small $^a$University of Primorska, UP IAM, Muzejski trg 2, 6000 Koper, Slovenia\\
	$^b$University of Primorska, UP FAMNIT, Glagolja\v ska 8, 6000 Koper, Slovenia
}

\end{center}

\begin{abstract}
The canonical double cover $B(X)$ of a graph $X$ is the direct product of $X$ and $K_2$. If $Aut(B(X)) \cong Aut(X) \times \mathbb{Z}_2$ then $X$ is called stable; otherwise $X$ is called unstable.  An unstable graph is nontrivially unstable if it is connected, non-bipartite and distinct vertices have different neighborhoods. Circulant is a Cayley graph on a cyclic group. Qin et al.  conjectured  in [J. Combin. Theory Ser. B 136 (2019), 154-169] that there are no nontrivialy unstable circulants of odd order. In this paper we prove this conjecture.
\end{abstract}

\begin{quotation}
\end{quotation}
\begin{quotation}
\end{quotation}
\section{Introduction}\label{sec:intro}
\noindent

All groups considered in this paper are finite and all graphs are finite, simple and undirected. For a graph $X$, we denote by $V(X)$, $E(X)$ and $\Aut(X)$ the vertex set, the edge set and the automorphism group of $X$, respectively. 
{\it The canonical double cover} (also called {\it the bipartite double cover} or {\it the Kronecker cover}) of a graph $X$, denoted by $B(X)$, is the direct product $X \times K_2$ (where $K_2$ denotes the complete graph on two vertices).  
This means that $V(B(X))=V(X)\times \ZZ_2$ and $E(B(X))=\{\{(x,0),(y,1)\}\mid \{x,y\}\in E(X)\}$. 
Canonical double covers have been studied by several authors, see for example \cite{FKMM, Krnc, NS96, W, Z}.
It is well-known that $B(X)$ is connected if and only if $X$ is connected and non-bipartite, see \cite{HIK}.
It can also be easily seen that $\Aut(B(X))$ contains a subgroup isomorphic to $\Aut(X)\times \ZZ_2$.  However, determining the full automorphism group of $B(X)$ is not as trivial. Hammack and Imrich \cite{HI} investigated vertex-transitivity of direct product of graphs, and proved that for a non-bipartite graph $X$ and a bipartite graph $Y$, direct product $X\times Y$ is vertex-transitive if and only if both $B(X)$ and $Y$ are vertex-transitive. Hence, the problem of vertex-transitivity of direct product of graphs reduces to the problem of vertex-transitivity of canonical double covers.
If $\Aut(B(X))$ is isomorphic to $\Aut(X)\times \ZZ_2$ then the graph $X$ is called {\it stable}, otherwise it is called {\it unstable}. 
This concept was first defined by Maru\v si\v c et al. \cite{MSS89} and studied later most notably by Surowski \cite{S01,S03}, Wilson \cite{W08}, Lauri et al. \cite{LMS15} and Qin et al. \cite{QXZ}. 
A graph is said to be {\em irreducible} (or {\em vertex-determining}) if distinct vertices have different neighbours, and {\em reducible} otherwise.
It is easy to see that the following graphs are all unstable: disconnected graphs, bipartite graphs with non-trivial automorphism group, and reducible graphs.
An unstable graph is said to be {\em nontrivialy unstable} if it is non-bipartite, connected, and irreducible.

For a group $G$ and an inverse closed subset $S\subseteq G\setminus \{1_G\}$, the {\it Cayley graph} $Cay(G,S)$ on $G$ with connection set $S$ is defined as the graph with vertex set $G$, with two vertices $x,y\in G$ being adjacent if and only if $x^{-1}y\in S$. A Cayley graph on a cyclic group is called {\em circulant}. Sabidussi \cite{Sab} proved that a graph $X$ is a Cayley graph on a group $G$ if and only if $\Aut(X)$ has a regular subgroup isomorphic to $G$. It is easy to see that if $X$ is a Cayley graph on a group $G$, then its canonical double cover $B(X)$ is a Cayley graph on $G\times \ZZ_2$. The converse is not true in general, that is $B(X)$ can be a Cayley graph even if the graph $X$ is not Cayley. 
The problem of characterizing graphs with Cayley canonical double covers was studied by Maru\v si\v c, Scapellato and Zagaglia Salvi \cite{MSS92}, who introduced the class of generalized Cayley graphs (see also \cite{HKM15, HKPT17}) and proved that if the canonical double cover of a generalized Cayley graph
X is a stable graph, then X is a Cayley graph. The characterization of graphs whose canonical double covers are Cayley graphs was given by the second author in \cite{H19}.  

The problem of characterization of vertex-transitive unstable graphs was posed in \cite[Problem 5.7]{DMS19}. However, the problem is difficult even when restricted to the class of circulant graphs. The problem of classification of nontrivially  unstable circulants was posed in \cite[Problem 1.2]{QXZ}. Qin et al proved that every circulant of odd prime order is stable (see \cite[Theorem 1.4]{QXZ}, and posed the following conjecture.
\begin{conjecture}\cite[Conjecture 1.3]{QXZ}\label{the conjecture}
There is no nontrivially unstable circulant of odd order.
\end{conjecture}

In this paper we prove Conjecture~\ref{the conjecture}, that is we prove the following theorem.

\begin{theorem}\label{thm:mainTHM}
Let $X$ be a connected irreducible circulant of odd order and let $B(X)$ be its canonical double cover. Then $\Aut(B(X))\cong \Aut(X)\times \ZZ_2$.
\end{theorem}

\begin{remark}
Using Theorem~\ref{thm:mainTHM}, the automorphism group of a canonical double cover of any circulant $X$ of odd order can be determined in terms of the autormorphism group of $X$. Namely, if $X$ is a connected reducible circulant, then there exist a positive integer $d>1$ and a connected irreducible circulant $Y$, such that  $X\cong Y\wr \overline{K_d}$.
It is now easy to see that $B(X)\cong (B(Y)\wr \overline{K_d})$. Using the result of Sabidussi \cite{S59} on the automorphism group of wreath product of graphs, we obtain $Aut(B(X))\cong \Aut(B(Y)\wr S_d\cong (Aut(Y)\times \ZZ_2)\wr S_d$. 
\end{remark}

\section{Preliminaries}

Following \cite{PS}, an automorphism $\gamma$ of a bipartite graph $Y$ is called a {\em strongly switching involution}  if $\gamma$ is an involution that swaps the two colour
classes and fixes no edge. 
For a graph $X$ and a partition   $\mathcal{P}$ of $V(X)$, the {\it quotient graph of $X$ with respect to $\mathcal{P}$} is the graph whose vertex set is $\mathcal{P}$ with  $B_1,B_2\in \mathcal{P}$ being adjacent if and only if there exist $x\in B_1$ and $y\in B_2$ such that $\{x,y\}\in E(X)$.
The following result is proved in \cite[Proposition 2.4]{PS}.

\begin{lemma}\label{lem:strongly switching involution}
A bipartite graph $Y$ is isomorphic to the canonical double cover of a graph $X$ if and only if $\Aut(Y)$ contains a strongly switching involution $\gamma$ such that $X$ is isomorphic to the quotient graph of $Y$ with respect to the orbits of $\langle \gamma \rangle$.
\end{lemma}

The following lemma gives a necessary and sufficient condition for the stability of a graph, and it has been proved in \cite[Corollary 2.5]{H19}. However, we give the proof here for the sake of completeness.
\begin{lemma}\label{lem:tau central}
Let $X$ be a connected non-bipartite graph, and let $B(X)$ be its canonical double cover. Let $\tau$ be the automorphism of $B(X)$ defined by $\tau(x,i)=(x,i+1)$. Then $X$ is stable if and only if $\tau$ is central in $\Aut(B(X)$.
\end{lemma}
\begin{proof}
Suppose first that $X$ is stable. Then $\Aut(B(X))\cong \Aut(X)\times \langle \tau \rangle$, hence $\tau$ is central in $\Aut(B(X))$. Suppose now that $\tau$ is central in $\Aut(B(X))$. Let $\alpha\in \Aut(B(X))$ be arbitrary. Suppose first that $\alpha$ fixes the color classes of $B(X)$. Let $g$ be the permutation of $V(X)$ such that $\alpha(x,0)=(g(x),0)$. Since $\tau$ commutes with $\alpha$ it follows that $\alpha(x,1)=\alpha (\tau (x,0))=\tau (\alpha(x,0))=(g(x),1)$. Let $\{x,y\}\in E(X)$. Then $\{(x,0),(y,1)\}\in E(B(X))$, and since $\alpha \in \Aut(B(X))$ it follows that $\{(g(x),0),(g(y),1)\}\in E(B(X))$. By definition of the canonical double cover, it  follows that $\{g(x),g(y)\}\in E(X)$. It follows that $g\in \Aut(X)$. Hence, $\alpha$ is induced by the automorphism of $X$, that is $\alpha\in \Aut(X)\times \ZZ_2$. If $\alpha$ permutes the color classes of $B(X)$, then we apply the above arguments for $\alpha \tau$ which fixes the color classes, and conclude that $\alpha \tau \in \Aut(X)\times \ZZ_2$, hence $\alpha\in \Aut(X)\times \ZZ_2$. This finishes the proof.
\end{proof}

Suppose that $X=Cay(\ZZ_k,S)$ is a circulant of odd order $k$. Then it is straightforward to verify that  $B(X)$ is a Cayley graph on $\ZZ_{k}\times \ZZ_2$, with connection set $S\times \{1\}$. Since the map $\alpha: \ZZ_k \times \ZZ_2 \to \ZZ_{2k}$ defined by $\alpha(x,i)=2x+ki$ is an isomorphism, it follows that $B(X)\cong \Cay(\ZZ_{2k},k+2S)$ (where elements of $S$ are now considered as elements of $\ZZ_{2k}$), hence $B(X)$ is a circulant graph of order $2k$. Observe that the mapping $k_L:\ZZ_{2k}\to \ZZ_{2k}$ defined by $k_L(x)=x+k$ is the automorphism of $B(X)$ that corresponds to the map $\tau$ from Lemma~\ref{lem:tau central}. Hence, for circulants of odd order we have the following result.

\begin{lemma}\label{lem:central}
Let $X=Cay(\ZZ_k,S)$ be a connected circulant of odd order $k$. Then $X$ is stable if and only if the permutation $k_L$ is central in the automorphism group of $Cay(\ZZ_{2k},k+2S)= B(X)$.
\end{lemma}

We will now give another characterization of stable vertex-transitive graphs.

\begin{lemma}
Let $X$ be a connected non-bipartite vertex-transitive graph. Let $B(X)$ be the canonical double cover of $X$, and let $A=Aut(B(X))$. Then $X$ is stable if and only if $A_{(v,0)}=A_{(v,1)}$ for some $v\in V(X)$.
\end{lemma}
\begin{proof}
Suppose first that $X$ is stable. Then by Lemma~\ref{lem:tau central} it follows that the map $\tau$ is central in $A$. If $\varphi\in A_{(v,0)}$ then we have $\varphi(v,1)=\varphi(\tau(v,0))=\tau(\varphi(v,0))=(v,1)$. This shows that $A_{(v,0)}\leq A_{(v,1)}$. Since $B(X)$ is vertex-transitive, all stabilizers are of the same order, hence $A_{(v,0)}=A_{(v,1)}$.

Suppose now that $A_{(v,0)}=A_{(v,1)}$ for some $v\in V(X)$. Let $w\in V(X)$ be arbitrary, let $g$ be an automorphism of $X$ that maps $v$ into $w$, and let $\beta$ be the automorphism of $B(X)$ defined by $\beta(x,i)=(g(x),i)$. Then $A_{(w,0)}=\beta A_{(v,0)} \beta^{-1}=\beta A_{(v,1)} \beta^{-1}=A_{(w,1)}$. 

Let $\alpha$ be an automorphism of $B(X)$. We will prove that $\alpha$ and $\tau$ commute. Without loss of generality, we may assume that $\alpha$ fixes the color classes of $B(X)$. Suppose that $\alpha(w,0)=(u,0)$, for some $u\in V(X)$. Let $h$ be an automorphism of $X$ that maps $u$ to $w$, and let $\gamma$ be the automorphism of $B(X)$ defined by $\gamma(x,i)=(h(x),i)$. Observe that $\gamma \alpha\in A_{(w,0)}=A_{(w,1)}$, hence $\gamma(\alpha(w,1))=(w,1)$. It follows that $\alpha(w,1)=\gamma^{-1}(w,1)=(u,1)$. It is now straightforward to conclude that $\alpha\tau (w,0)=\tau \alpha (w,0)$. Since $w\in V(X)$ was arbitrary, it follows that $\alpha$ and $\tau$ commute. We conclude that $\tau$ is central in $\Aut(B(X))$. By Lemma~\ref{lem:tau central} it follows that $X$ is stable.
\end{proof}

For circulants, the above lemma implies the following.
\begin{lemma}\label{lem:equalstabilizers0m}
Let $X$ be a connected circulant of odd order $k$, let $B(X)$ be its canonical double cover and $A=Aut(B(X))$. Then $X$ is stable if and only if $A_0=A_k$.
\end{lemma}

A Cayley graph $\Gamma=Cay(G,S)$ is said to be {\em normal} if $G_L$ is a normal subgroup of $\Aut(\Gamma)$, or equivalently if $\Aut(\Gamma)_0=\Aut(G,S)$, where $\Aut(G,S)=\{ \varphi\in \Aut(G) \mid \varphi(S)=S\}$.

\begin{lemma}\label{lem:normal are stable}
Let $G$ be an abelian group of odd order, and let $X=Cay(G,S)$ be a connected Cayley graph on $G$. If $B(X)=Cay(G\times \ZZ_2,S')$ is a normal Cayley graph then $X$ is stable.
\end{lemma}
\begin{proof}
Since $B(X)$ is normal Cayley graph, it follows that each element of $\Aut(B(X))$ is a composition of some element of $(G\times \ZZ_2)_L$ with some element of $\Aut(G\times \ZZ_2)$. Let $\tau$ be the automorphism of $B(X)$ defined by $\tau(x,i)=(x,i+1)$. Let $t$ be the unique element of order $2$ in $G\times \ZZ_2$. Observe that $\tau=t_L$. Since $G$ is abelian, it follows that each element of $(G\times \ZZ_2)_L$ commutes with $\tau$. Let $\varphi\in \Aut(G\times \ZZ_2)$. Then $\varphi(t)=t$, since group automorphisms preserve the order of elements, and $t$ is the unique element of order $2$ in $G\times \ZZ_2$.
It follows that $(\varphi t_L)(x)=\varphi(tx)=t\varphi(x)=(t_L\varphi)(x)$. This shows that $\tau=t_L$ commutes with every element of $\Aut(G)$. The result now follows by Lemma~\ref{lem:central}.
\end{proof}

The following lemma tells that a normal circulant of even order not divisible by four has a unique regular cyclic subgroup.
\begin{lemma}\cite[Theorem 5.2.2]{XuPhD}\label{lem:normal cyclic regular}
Let $k$ be an odd positive integer, and let $X=Cay(\ZZ_{2k},S)$. Let $A=Aut(X)$ admitting a normal cyclic regular subgroup $H$. Then $H$ is the unique regular cyclic subgroup contained in $A$. 
\end{lemma}

The {\em wreath (lexicographic) product} $\Sigma \wr \Gamma$ of a graph $\Gamma$ by
a graph $\Sigma$ is the graph with vertex set $V(\Sigma)\times V(\Gamma)$ such that $\{(u_1,u_2),(v_1,v_2)\}$ is an edge
if and only if either $\{u_1,v_1\}\in E(\Sigma)$, or $u_1=v_1$ and $\{u_2,v_2\}\in E(\Gamma).$ 
Observe that $\Sigma \wr \Gamma$ is the graph obtained by substituting a copy of $\Gamma$ for each vertex of $\Gamma$.

The {\em deleted wreath (deleted lexicographic) product} of a graph $\Sigma$ and $\overline{K_d}$, denoted by $\Sigma \wr_d \ov{K_d}$, is the graph with vertex set $V(\Sigma)\times \ZZ_d$, such that $\{(u_1,i),(v_1,j)\}$ is an edge
if and only if $\{u_1,v_1\}\in E(\Sigma)$ and $i\neq j$.
Observe that $\Sigma \wr_d \ov{K_d}$ can be obtained from $\Sigma \wr \ov{K_d}$ by removing $d$ disjoint copies of $\Sigma$.
Observe that the canonical double cover of a graph $X$ is isomorphic to $X\wr_d \ov{K_2}$ (see \cite[Example 2.1]{QXZ}). The following result gives a characterization of all arc-transitive circulants.

\begin{lemma}
\label{lem:KovacsLi}
{\rm\cite{K03,Li05}}
Let $\Gamma$ be a connected arc-transitive circulant of order $n$. Then one of the following holds:
\begin{enumerate}
\itemsep=0pt
\item[(i)] $\Gamma \cong K_n;$
\item[(ii)] $\Gamma =\Sigma \wr \overline{K_d},$ where $n=md$, $m,d>1$ and $\Sigma$ is a connected arc-transitive circulant of order $m$;
\item[(iii)] $\Gamma=\Sigma \wr_d \overline{K_d}$ where $n=md,$ $d>3,$ $gcd(d,m)=1$ and $\Sigma$ is a
connected arc-transitive circulant of order $m$;
\item[(iv)] $\Gamma $ is a normal circulant.
\end{enumerate}
\end{lemma}

The following result is a direct consequence of \cite[Theorem 5.3]{DMS19}.
\begin{lemma}\label{lem:TedStefko}
Let $d\geq 3$ be an integer, let $\Sigma$ be an irreducible vertex-transitive graph whose order is not divisible by $d$, and let $\Gamma=\Sigma \wr_d \ov{K_d}$. Then $\Aut(\Gamma)\cong \Aut(\Sigma) \times S_d$.
\end{lemma}
\begin{proof}
Since $\Sigma$ is vertex-transitive, and $\Aut(\Gamma)$ contains a subgroup isomorphic to $\Aut(\Sigma)\times S_d$, it follows that $\Gamma$ is vertex-transitive. Since $\Sigma$ is reducible, by \cite[Theorem 5.3]{DMS19} it follows that $\Aut(\Gamma)\not \cong \Aut(\Sigma) \times S_d$ if and only if the conditions $(i)-(iii)$ of \cite[Theorem 5.3]{DMS19} hold. However, since $|V(\Sigma)|$ is not divisible by $d$, it follows that condition $(ii)$ of \cite[Theorem 5.3]{DMS19} doesn't hold. We conclude that $\Aut(\Gamma)\cong \Aut(\Sigma) \times S_d$.
\end{proof}

\begin{lemma}\cite[Lemma 2.3]{QXZ}\label{lem:worhtydoublecover}
The canonical double cover $B(X)$ of a graph $X$ is reducible if and only if $X$ is reducible.
\end{lemma}

The following result can be extracted from the proof of \cite[Proposition 2.1]{NW72}. However, as the statement of \cite[Proposition 2.1]{NW72} is a bit weaker, we include the proof here for the sake of completeness.

\begin{lemma}\label{lem:stabilizer of a set fixes a subgroup}
Let $X=Cay(G,S)$, $A=Aut(X)$ and let $K\subset S$ such that $\varphi(K)=K$ for every $\varphi \in A_1$. Then $\varphi(\langle K \rangle)=\langle K \rangle$ for every $\varphi \in A_1$. Moreover, if $K$ is inverse closed, then $\varphi$ induces an automorphism of $Cay(\langle K \rangle,K)$.
\end{lemma}
\begin{proof}
Let $x\in G$ and $\omega\in A_x$ be arbitrary, and let $M$ be any subset of $G$ such that $\varphi(M)=M$ for every $\varphi \in A_1$. Observe that $(x_L)^{-1}\omega x_L\in A_1$. By the assumption on $A_1$ and $M$ it follows that  $(x_L)^{-1}\omega x_L (M)=M$. Therefore, $\omega(xM)=xM$ for every $x\in G$ and every $\omega\in A_x$.

We claim that $\varphi(K^t)=K^t$ for every positive integer $t$. The proof is by induction on $t$. If $t=1$, the claim follows from the hypothesis. Suppose that $\varphi(K^t)=K^t$.
Let $\varphi \in A_1$ and let $k_1\in K$ be arbitrary. Observe that $(k_1k_2^{-1})_L\varphi \in A_{k_1}$. Applying the first part of the proof with $M=K^t$ and $x=k_1$, it follows that $(k_1k_2^{-1})_L\varphi (k_1K^t)=k_1K^t$. This implies that $\varphi(k_1K^t)=k_2K^t=\varphi(k_1)K^t\subseteq K^{t+1}$. Since this is true for every $k_1\in K$, it follows that $\varphi(K^{t+1})=K^{t+1}$. 

Since $\langle K \rangle=K\cup K^2 \cup \ldots \cup K^t$ for some positive integer $t$, it follows that $\varphi(\langle K \rangle)=\langle K \rangle$. Moreover, let $e$ be an arbitrary edge of $Cay(\langle K \rangle,K)$. Then $e=\{k_1,k_1k_2\}$ for some $k_1,k_2\in K$. As we proved that $\varphi(k_1K^t)=\varphi(k_1)K^t$ for every positive integer $t$, it follows that $\varphi(k_1K)=\varphi(k_1)K$. Since $\varphi(k_1k_2)\in \varphi(k_1K)=\varphi(k_1)K$ we conclude that $\varphi(k_1k_2)=\varphi(k_1)k$ for some $k\in K$. This shows that $\varphi(e)$ is an edge of $Cay(\langle K \rangle,K)$, hence $\varphi$ induces an automorphism of $Cay(\langle K \rangle,K)$.
\end{proof}

\section{Main result}

The following lemma gives a partial generalization of \cite[Theorem 1.6]{QXZ}, where it is proved that there is no arc-transitive nontrivally unstable circulant. It is easy to see that $B(X)$ is arc-transitive if $X$ is arc-transitive. 
\begin{lemma} \label{lem:arc-transitive}
Let $X$ be a nontrivialy unstable circulant of odd order $m$. Then $B(X)$ is not arc-transitive. 
\end{lemma}
\begin{proof}
Suppose that $B(X)$ is an arc-transitive circulant. Since $B(X)$ is also a connected circulant, by Lemma~\ref{lem:KovacsLi}, it follows that $B(X)$ is a complete graph, normal circulant, wreath product or a deleted wreath product. 
If $B(X)$ is a complete graph, then it must be isomorphic to $K_1$ or $K_2$, since they are the only bipartite complete graphs. However, it is easy to see that none of them is a canonical double cover of some graph. Lemma~\ref{lem:normal are stable} implies that $B(X)$ is not a normal circulant. If $B(X)$ is a wreath product, then $B(X)$ is reducible, hence by Lemma~\ref{lem:worhtydoublecover} it follows that $X$ is also reducible, contrary to the assumption that $X$ is nontrvialy unstable.
Therefore, we may assume that $B(X)$ is the deleted wreath product. Suppose that $B(X) \cong \Sigma \wr_d \overline{K_d}$, where $2m = td$, $d > 3$, $gcd(d,t) = 1$ and $\Sigma$ is a connected
arc-transitive circulant of order $t$. 
Since $B(X)$ is bipartite and $d>3$ it follows that $\Sigma$ is also bipartite, hence order of $\Sigma$ is even.

If $\Sigma$ is reducible, then $\Sigma\cong \Sigma_1\wr \ov{K_{d_1}}$, where $\Sigma_1$ is an irreducible circulant. Then by \cite[Proposition 4.5]{DMS19}, it follows that $B(X)\cong \Sigma \wr_d \overline{K_d}\cong (\Sigma_1 \wr_d \ov{K_d})\wr \ov{K_{d_1}}$, implying that $B(X)$ is reducible. By Lemma~\ref{lem:worhtydoublecover} it follows that $X$ is also reducible, a contradiction.

Suppose now that $\Sigma$ is irreducible. Then by Lemma~\ref{lem:TedStefko} it follows that $Aut(B(X))\cong \Aut(\Sigma) \times S_d$. Recall that $\Sigma$ is a connected arc-transitive bipartite circulant of even order. 
If $\Sigma$ is a wreath product, then again $B(X)$ is a wreath product, hence $B(X)$ is reducible, and consequently also $X$ is reducible, contrary to the assumption that $X$ is nontrivialy unstable. We conclude that $\Sigma$ is a normal circulant, or $\Sigma\cong \Sigma_1 \wr_d \ov{K_{d_1}}$, where $\Sigma_1$ is an arc-transitive circulant. Applying the same arguments for $\Sigma_1$, we conclude that $\Sigma_1$ is a normal circulant, or a deleted wreath product. As $\Gamma$ is finite, this process has to terminate, so eventually we will get $B(X)=(\Sigma_t \wr_d \ov{K_{d_t}})\wr_d \ldots \wr_d \ov{K_d}$, where $\Sigma_t$ is a normal circulant of even order and $\Aut(B(X))\cong \Aut(\Sigma_t) \times S_{d_t}\times \ldots \times S_d$. 

Observe that $S_{d_t}\times \ldots \times S_d$ is a normal subgroup of $\Aut(B(X))$, hence its orbits form a system of imprimitivity for $\Aut(B(X))$. Let $\mathcal{P}$ denote the set of orbits of $S_{d_t}\times \ldots \times S_d$ on $V(B(X))$. Observe that the quotient graph of $B(X)$ with respect to $\mathcal{P}$ is isomorphic to $\Sigma_t$.
 The group $(\ZZ_{2m})_L$ projects into cyclic regular subgroup of $\Aut(\Sigma_t)$. For $g\in \Aut(B(X))$, let $g/_{\mathcal{P}}$ denote the permutation induced by the action of $g$ on $\mathcal{P}$.
Since $\Sigma_t$ is a normal circulant, by Lemma~\ref{lem:normal cyclic regular} it follows that $(\ZZ_{2m})_L/_{\mathcal{P}}$ is a normal cyclic regular subgroup of $\Aut(\Sigma_t)$, hence $m_L/_{\mathcal{P}}$ is central in  $\Aut(\Sigma_t)$. If follows that $m_L$ is central in $\Aut(\Sigma_t) \times S_{d_t}\times \ldots \times S_d$.
The result now follows by Lemma~\ref{lem:central}.
\end{proof}

In the following lemma, we consider graphs that can be realised as a canonical double cover of some  connected arc-transitive circulant of odd order, and derive certain important properties of their automorphism groups.

\begin{lemma}\label{lem:key lemma}
 Let $m$ be an odd positive integer, and let $\Gamma$ be a connected bipartite arc-transitive circulant of order $2m$ and even  valency, and let $A=Aut(\Gamma)$.
 Then one of the following holds:
 \begin{enumerate}[(i)]
 \item $A_0=A_m$, or
 \item $\Gamma \cong \Gamma_1 \wr \ov{K_d}$ where $\Gamma_1$ is an irreducible arc-transitive circulant of even order $2m_1$ and $\Aut(\Gamma_1)_0=\Aut(\Gamma_1)_{m_1}$.
 \end{enumerate}
 \end{lemma}
 \begin{proof}
Since $\Gamma$ is a circulant of order $2m$ and has even valency, it follows that $\Gamma\cong Cay(\ZZ_{2m},S)$ where $S$ does not contain element $m$. This implies that the automorphism $m_L:x \mapsto m+x$ is a strongly switching involution of $\Gamma$. By Lemma~\ref{lem:strongly switching involution}, it follows that $\Gamma$ is the canonical double cover of the graph $X$ obtained as the quotient graph of $\Gamma$ with respect to the orbits of $m_L$. Observe that $X$ is a connected and non-bipartite circulant of order $m$, since $\Gamma$ is a connected circulant of order $2m$.

If $X$ is a stable graph, then by Lemma~\ref{lem:equalstabilizers0m} it follows that $\Gamma$ satisfies condition $(i)$, and we are done.
As $\Gamma$ is arc-transitive, by Lemma~\ref{lem:arc-transitive} it follows that $X$ is not nontrivialy unstable.  
Therefore, $X$ is trivially unstable, and since it is connected and non-bipartite, it follows that $X$ is reducible. We conclude that $X\cong \Sigma \wr \ov{K_d}$, where $\Sigma$ is a connected irreducible circulant of order $m_1$, with $m_1$ being odd.
 
Observe that $\Gamma\cong X\times K_2\cong (\Sigma \wr \ov{K_d}) \times K_2\cong (\Sigma \times K_2)\wr \ov{K_d}$. Let $\Gamma_1=\Sigma \times K_2$.
Since $\Gamma$ is an arc-transitive circulant, it follows that $\Gamma_1$ is an arc-transitive circulant (see \cite[Remark 1.2]{LMM}). Since $\Gamma_1$ is a canonical double cover of $\Sigma$, by Lemma~\ref{lem:arc-transitive} it follows that $\Sigma$ is not non-trivially unstable. Recall that $\Sigma$ is irreducible, connected and non-bipartite. We conclude that $\Sigma$ is stable, hence by Lemma~\ref{lem:equalstabilizers0m} it follows that $\Aut(\Gamma_1)_0=\Aut(\Gamma_1)_{m_1}$.
 \end{proof}

We are now ready to prove the main result of this paper. We will show that there is no non-trivially unstable circulant of odd order.

\begin{sloppypar}
\begin{theorem3}
Let $X$ be a connected irreducible circulant of odd order and let $B(X)$ be its canonical double cover. Then $\Aut(B(X))\cong \Aut(X)\times \ZZ_2$.
\end{theorem3}
\end{sloppypar}
 \begin{proof}
 Let $X=Cay(\ZZ_{m},S)$ be a connected irreducible circulant of odd order $m$. If $X$ is stable, the result follows by the definition. It is clear that $X$ is non-bipartite, as it is of odd order. Hence, we may assume that $X$ is a nontrivially unstable. We have that $B(X)=Cay(\ZZ_{2m},S')$ is a circulant of order $2m$, where $S'=m+2S$. Let $A=\Aut(B(X))$. If $A_0$ is transitive on $S'$, then $B(X)$ is arc-transitive, and the result follows by Lemma~\ref{lem:arc-transitive}. Let $S_1,\ldots,S_k$ be the orbits of $A_0$ on $S'$. Observe that $S_i=-S_i$ (since mapping $i:x\mapsto -x$ is contained in $A_0$), and $m\not \in S_i$, for every $i\in \{1,\ldots,k\}$. Let $\Gamma_i= Cay(\langle S_i \rangle,S_i)$. Observe that $\langle S_i \rangle$ is a subgroup of $\ZZ_{2m}$ of even order, hence it contains the element of order $2$, that is $m\in \langle S_i\rangle$ for every $i\in \{1,\ldots,k\}$.
By Lemma~\ref{lem:stabilizer of a set fixes a subgroup} every element of $A_0$ fixes $\Gamma_i$, hence it follows that every element of $A_0$ induces an automorphism of $\Gamma_i$. As $A_0$ acts transitively on $S_i$ it follows that $\Gamma_i$ is arc-transitive. Therefore, $\Gamma_i$ is an arc-transitive circulant of even order (not divisible by 4) and even valency. Hence we can apply Lemma~\ref{lem:key lemma} to each of the graphs $\Gamma_i$. If for some $i$ we have that $\Gamma_i$ satisfies condition $1$ of Lemma~\ref{lem:key lemma}, it follows that every automorphism of $B(X)$ that fixes $0$, must also fix $m$, hence $A_0=A_m$ and by Lemma~\ref{lem:equalstabilizers0m} it follows that $X$ is stable, a contradiction.

We can now assume that $\Gamma_i\cong \Sigma_i\wr \ov{K_{d_i}}$, where $\Sigma_i$ is an irreducible arc-transitive circulant of even order satisfying condition $(i)$ from Lemma~\ref{lem:key lemma}. 
Let $R$ be the equivalence relation of ``having the same neighbourhood'' defined on $V(\Gamma_i)$. The equivalence classes of this relation are all of size $d_i$, and form a system of imprimitivity for $\Aut(\Gamma_i)$. Let $H_i$ be the kernel of the action of $\langle S_i \rangle_L$ (which is a regular cyclic subgroup of $\Aut(\Gamma_i)$) on the partition induced by $R$. The permutation group $\langle S_i \rangle_L/H_i$ induced by the action of $\langle S_i \rangle_L$ on the classes of $R$ is a cyclic regular group.
Since $H_i$ is semiregular on $V(\Gamma_i)$, it follows that the equivalence classes of $R$ coincide with the orbits of $H_i$. 
It follows that  $S_i$ is a union of cosets of the subgroup $H_i$ of $\langle S_i \rangle$. Observe that the element of order $2$ in the quotient group $\langle S_i \rangle / H_i$ is $m+H_i$.

Recall that $\Sigma_i$ has the property that the stabilizer of the identity and the element of order $2$ in the cyclic regular subgroup of $\Aut(\Sigma_i)$ are equal, which implies that every automorphism of $\Gamma_i$ that fixes $0$, fixes setwise coset $m+H_i$. As every automorphism of $\Gamma$ that fixes $0$ induces automorphism of $\Gamma_i$, it follows that every automorphism of $\Gamma$ that fixes $0$ fixes setwise $m+H_i$. If $GCD(d_1,\ldots,d_k)=d>1$, then $H=H_1\cap \ldots \cap H_k$ has order $d$. It follows that $S'$ is a union of cosets of $H$, hence  $B(X)$ is a wreath product with $\ov{K_d}$. This shows that  $B(X)$ is reducible, and by Lemma~\ref{lem:worhtydoublecover} it follows that $X$ is also reducible, contrary to the assumption that $X$ is nontrivially unstable circulant.

If $GCD(d_1,\ldots,d_k)=1$, then it follows that $H_1\cap \ldots \cap H_k=\{0\}$. As observed above, every automorphism of $\Gamma$ fixes setwise each of the sets $m+H_i$, for $i\in \{1,\ldots,k\}$, hence it also fixes their intersection. Since $(m+H_1) \cap \ldots \cap (m+H_k)=\{m\}$, by Lemma~\ref{lem:equalstabilizers0m} it follows that $X$ is stable, a contradiction. This finishes the proof.
 \end{proof}
 
 For further research, we propose the following problem.
 
 \begin{problem}
 Does there exist a nontrivially unstable Cayley graph on an Abelian group of odd order.
 \end{problem}
 
 \section*{Acknowledgement}
 This work has been supported by the Slovenian
Research Agency (Young researchers program, research programs P1-0404 and P1-0285, and research projects J1-1691, J1-1694, J1-1695,
J1-9110, N1-0102 and N1-0140).


\begin{thebibliography}{99}
 
\bibitem{BDM} S. Bhoumik, E. Dobson and Y. Morris, On the automorphism groups of almost all circulant graphs and digraphs, {\it Ars Math. Contemp.} {\bf 7} (2014), 487-506.


\bibitem{DMS19} T. Dobson, \v S. Miklavi\v c, and P.~ \v Sparl, On Automorphism Groups of Deleted Wreath Products. {\em Mediterr. J. Math.} {\bf 16}, 149 (2019).

\bibitem{FKMM} 
Y.Q.~Feng, K. Kutnar, A. Malni\v c and D. Maru\v si\v c, 
On 2-fold covers of graphs, {\em J. Combin. Theory Ser. B} {\bf 98} (2008), 324-341.

\bibitem{HI} R.~Hammack and W. Imrich, Vertex-transitive direct products of graphs, {\em Electron. J. Combin.} {\bf 25} (2018), Paper 2.10, 16 pp.

\bibitem{HIK} R.~Hammack, W.~Imrich and S.~Klav\v zar. {\em Handbook of Product Graphs,
Second Edition.} Discrete Mathematics and Its Applications. Taylor \&
Francis, 2011.

\bibitem{H19} A.~Hujdurovi\' c, Graphs with Cayley canonical double covers. {\em Discrete Math.} {\bf 342} (2019), no. 9, 2542-2548.

\bibitem{HKM15} A.~Hujdurovi\' c, K.~Kutnar and D.~Maru\v si\v c, Vertex-transitive generalized {C}ayley graphs which are not
              {C}ayley graphs,
               {\em European J. Combin.} {\bf 46} (2015), 45--50.
               
\bibitem{HKPT17} A.~Hujdurovi\' c, K.~Kutnar,  P.~Petecki and  A.~Tanana,  On automorphisms and structural properties of generalized Cayley graphs, {\em Filomat} {\bf 31} (2017), 4033-4040.

\bibitem{K03} I.~Kov\'acs,
	Classifying arc-transitive circulants ,
	{\em J. Algebr. Combin.} {\bf 20} (2004), 353--358.

\bibitem{Krnc} M. Krnc and T. Pisanski, Characterization of generalized Petersen graphs that are Kronecker covers, (2018), {\em arXiv: 1802.07134 [math. CO]}
               
\bibitem{LMS15} J.~Lauri, R.~Mizzi and R.~Scapellato,
 Unstable graphs: a fresh outlook via {TF}-automorphisms,
 {\em Ars Math. Contemp.} {\bf 8} (2015), 115-131.
 
\bibitem{Li05} C. H. Li, Permutation groups with a cyclic regular subgroup and arc transitive circulants, {\em J.
Algebraic Combin.} 21 (2005), 131–136.

\bibitem{LMM} C. H. Li, D. Maru\v si\v c, J. Morris, 
Classifying arc-transitive circulants of square-free order. 
{\em J. Algebraic Combin.} {\bf 14} (2001), no. 2, 145--151.

\bibitem{MSS89} D.~Maru\v si\v c, R.~Scapellato and N.~Zagaglia Salvi. A characterization
of particular symmetric (0, 1)-matrices. {\em Linear Algebra Appl.} {\bf 119} (1989), 153–-162.               
               
\bibitem{MSS92} D.~Maru\v si\v c, R.~Scapellato and N.~Zagaglia~Salvi,
    Generalized Cayley graphs,
    {\em Discrete Math.} {\bf 102} (1992), 279--285.               
               
\bibitem{NS96} R. Nedela and M. \v Skoviera, Regular embeddings of canonical double coverings of graphs, {\em J. Combin.
Theory Ser. B} {\bf 67} (1996), 249--277.

\bibitem{NW72}  L.A. Nowitz, M.E. Watkins, Graphical regular representations of non-abelain groups, I, {\em Canad. J. Math.} {\bf 24} (1972) 994--1008.

\bibitem{PS} W.~Pacco and R.~Scapellato,  Digraphs having the same canonical double covering, {\em Discrete Math.} {\bf 173} (1997), 291-296.

\bibitem{QXZ} Y.L. Qin, B. Xia and S. Zhou, Stability of circulant graphs, https://arxiv.org/abs/1802.04921.

\bibitem{S59} G. Sabidussi, The composition of graphs, {\em Duke Math J.} {\bf 26} (1959), 693--696.

\bibitem{Sab} G.~Sabidussi, On a class of fixed-point-free graphs, {\em Proc. Amer. Math.} Soc. {\bf 9} (1958), 800-804.


\bibitem{S01} D.~B.~Surowski, Stability of arc-transitive graphs, {\em J. Graph Theory}, {\bf 38} (2001), 95--110.

\bibitem{S03} D.~B.~Surowski, Automorphism groups of certain unstable graphs.
{\em Math. Slovaca}, {\bf 53} (2003), 215–-232.


\bibitem{XuPhD} Y.~ Xu (2018), {\em Normal and non-normal cayley graphs}, [Doctoral dissertation, The University of Western Australia],  \url{https://doi.org/10.26182/5c496c0ddd2da}

\bibitem{W} D. Waller, Double covers of graphs, {\em Bull. Austral. Math. Soc.} {\bf 14} (1976) 233--248.

\bibitem{W08} S.~Wilson, Unexpected symmetries in unstable graphs,
{\em J. Combin. Theory Ser. B} {\bf 98} (2008), 359--383.
    


\bibitem{Z}  B. Zelinka, On double covers of graphs, {\em Math. Slovaca} {\bf 32} (1982) 49-54.

\end{thebibliography}
\end{document}